\documentclass[12pt,reqno]{amsart}
\usepackage{indentfirst, amssymb, amsmath, amsthm, mathrsfs, setspace, indentfirst, enumerate,  mathrsfs, amsmath, amsthm}
\usepackage[bookmarksnumbered, colorlinks, plainpages]{hyperref}
\usepackage{mathrsfs}
\usepackage{cite}

\newtheorem*{theoA}{Theorem A}
\newtheorem*{theoB}{Theorem B}
\newtheorem*{theoC}{Theorem C}
\newtheorem*{theoD}{Theorem D}

\newtheorem*{cor A}{Corollary A}
\newtheorem*{cor B}{Corollary B}

\newtheorem{theo}{Theorem}[section]
\newtheorem{lem}{Lemma}[section]

\newtheorem{defi}{Definition}[section]

\newcommand{\ol}{\overline}
\newcommand{\be}{\begin{equation}}
\newcommand{\ee}{\end{equation}}
\newcommand{\beas}{\begin{eqnarray*}}
\newcommand{\eeas}{\end{eqnarray*}}
\newcommand{\bea}{\begin{eqnarray}}
\newcommand{\eea}{\end{eqnarray}}

\numberwithin{equation}{section}
\begin{document}
\title[C\MakeLowercase {oefficient Bounds and growth estimates for a Class of}......]{\LARGE C\Large\MakeLowercase {oefficient Bounds and growth estimates for a Class of Pluriharmonic Mappings in unit polydisk}}

\date{}
\author[S. M\MakeLowercase{ajumder}, G. H\MakeLowercase{aldar}, D. P\MakeLowercase{ramanik} \MakeLowercase{and} S. P\MakeLowercase{anja}]{S\MakeLowercase{ujoy} M\MakeLowercase{ajumder}$^1$, G\MakeLowercase{outam} H\MakeLowercase{aldar}$^2$ $^*$, D\MakeLowercase{ebabrata} P\MakeLowercase{ramanik}$^3$ \MakeLowercase{and} S\MakeLowercase{hantanu} P\MakeLowercase{anja}$^4$}

\address{$^1$Sujoy Majumder, Department of Mathematics, Raiganj University, Raiganj, West Bengal-733134, India.}
\email{sm05math@gmail.com, sjm@raiganjuniversity.ac.in}

\address{$^2$Goutam Haldar, Department of Mathematics, Ghani Khan Choudhury Institute of Engineering and Technology, Narayanpur, Malda-732141, West Bengal, India.}
\email{goutamiitm@gmail.com, goutamiit1986@gmail.com}

\address{$^3$Debabrata Pramanik, Department of Mathematics, Raiganj University, Raiganj, West Bengal-733134, India.}
\email{debumath07@gmail.com}

\address{$^4$Shantanu Panja, Department of Mathematics, University of Kalyani, West Bengal 741235, India.}
\email{panjasantu07@gmail.com}

\renewcommand{\thefootnote}{}
\footnote{*\emph{Corresponding Author}: Goutam Halder.}

\renewcommand{\thefootnote}{\arabic{footnote}}
\setcounter{footnote}{0}

\begin{abstract}
In this paper, we first introduce and study the class $\mathcal{P}_{\mathcal{H}_n^0}(M)$ of normalized pluriharmonic mappings, characterized by a specific bound on the sum of their second-order partial derivatives. We prove a one-to-one correspondence between this pluriharmonic class and a class of holomorphic functions, extending the known result of  Ghosh and Vasudevarao \cite{Ghosh-Allu-2020} to the setting of several complex variables.  Finally, we provide sharp coefficient bounds and growth estimates for functions in the class $\mathcal{P}_{\mathcal{H}_n^0}(M)$. 
\end{abstract}
\thanks{Typeset by \AmS -\LaTeX}
\maketitle

\section{\bf Introduction}
Pluriharmonic mappings are a natural and powerful generalization of harmonic functions from one complex variable to multiple complex variables. In simple terms, a real-valued function $u$ on a domain in $\mathbb{C}^n$ is pluriharmonic if its restriction to any complex line is harmonic.

Because they sit squarely at the intersection of complex analysis, differential geometry, and mathematical physics, they have several elegant applications.

In the theory of surfaces, harmonic mappings are famously used to parameterize minimal surfaces (surfaces with zero mean curvature, like soap films) via the Weierstrass-Enneper representation.

$\bullet$ Higher Dimensions: Pluriharmonic mappings allow mathematicians to extend these concepts to higher dimensions. They are used to study and construct pluriharmonic maps from K\"{a}hler manifolds into Riemannian manifolds.

$\bullet$ Hermitian Geometry: They help characterize specific geometric structures, such as identifying when a metric on a complex manifold can be considered K\"{a}hler or conformally K\"{a}hler.

Pluriharmonic functions serve as the ``backbone'' of potential theory in higher dimensions.

$\bullet$ Levy Problems and Pseudoconvexity: They are deeply tied to the study of plurisubharmonic functions, which define the boundaries of domains of holomorphy (pseudoconvex domains).

$\bullet$ Dirichlet Problem: They are used to solve boundary value problems on complex manifolds, helping to understand how holomorphic functions behave near the boundaries of high-dimensional shapes.

In theoretical physics, particularly in sigma models and string theory, fields are often modeled as maps between geometric spaces.
Supersymmetry and K\"{a}hler Geometry: Harmonic and pluriharmonic maps appear when physicists look for energy-minimizing configurations (ground states) of strings moving through a curved spacetime target space, especially when that target space has a K\"{a}hler structure.

In applied mathematics, pluriharmonic mappings find a home in the mechanics of materials.
Deformation Modeling: Because pluriharmonic maps preserve certain structural and geometric alignments, they are used to model ideal, energy-minimizing deformations in multi-dimensional elastic bodies where complex structures naturally arise.

\subsection{{\bf Basic Notations in several complex variables}}
The absolute value of a complex number $z_1$ is denoted by $|z_1|$ and for $z=(z_1,\ldots,z_n)\in\mathbb{C}^n$, we define
\begin{align*}
||z||^2=\sum\limits_{k=1}^n|z_k|^2\quad \text{and}\quad ||z||_{\infty}=\max\limits_{1\leq i\leq n}|z_i|.
\end{align*}

\smallskip
A multi-index $\alpha=(\alpha_1,\ldots,\alpha_n)$ of dimension $n$ consists of n non-negative integers $\alpha_j,\;1\leq j\leq n$; the degree of a multi-index $\alpha$ is the sum $|\alpha|=\sum_{j=1}^n \alpha_j$ and we denote $\alpha!=\alpha_1!\ldots \alpha_n!$. For $z=(z_1,\ldots,z_n)\in\mathbb{C}^n$ and a multi-index $\alpha=(\alpha_1,\ldots,\alpha_n)$, we define 
\[z^{\alpha}=\prod\limits_{j=1}^n z_j^{\alpha_j}\;\;\text{and}\;\;|z|^{\alpha}=\prod\limits_{j=1}^n |z_j|^{\alpha_j}.\]

\begin{defi}
An open polydisk \emph{(or open polycylinder)} in $\mathbb{C}^n$ is a subset $\mathbb{P}\Delta(a;r)\subset \mathbb{C}^n$ of the form 
\[\mathbb{P}\Delta(a;r)=\prod\limits_{j=1}^n \Delta(a_j;r_j)=\lbrace z\in\mathbb{C}^n: |z_i-a_i|<r_i,\;i=1,2,\ldots,n\rbrace,\]
the point $a=(a_1,\ldots,a_n)\in\mathbb{C}^n$ is called the centre of the polydisk and $r=(r_1,\ldots,r_n)\in\mathbb{R}^n\;(r_i>0)$ is called the polyradius. It is easy to see that
\begin{align*}
\mathbb{P}\Delta(0;1)=\mathbb{P}\Delta(0_n;1_n)=\prod\limits_{j=1}^n \Delta(0_j;1_j).
\end{align*}
\end{defi}

The unit disk in the complex plane is denoted by $\mathbb{D}$.\vspace{1.2mm}

\begin{defi} If $f(z)=f(z_1,\ldots,z_n)$ is continuous in $\Omega\subset \mathbb{C}^n$ and holomorphic in each variable $z_k, k=1,\ldots,n$, separately, then $f(z)$ is said to be holomorphic in $\Omega$. We also call $f(z)=f(z_1,\ldots,z_n)$ a holomorphic function of $n$ variables $z_1,\ldots,z_n$.
\end{defi}
We write
\begin{align*}
z_j = x_j + i y_j \qquad (i^2 = -1,\; j = 1,\dots,n),
\end{align*}
where \(x_j\) and \(y_j\) are real numbers. We set
\begin{align*}
f(z) = u(x,y) + i v(x,y),
\end{align*}
where $u(x,y)$ and $v(x,y)$ are the real and imaginary parts of $f(z)$; $x = (x_1,\dots,x_n)$ and $y = (y_1,\dots,y_n)$.
The Cauchy--Riemann equations for each $z_j\;(j=1,\dots,n)$ are
\begin{align}\label{Eq 1.1}
\frac{\partial u}{\partial x_j}
=
\frac{\partial v}{\partial y_j}\quad \text{and}\quad
\frac{\partial u}{\partial y_j}
=
-\,\frac{\partial v}{\partial x_j}
\qquad (j=1,\dots,n).
\end{align}

By differentiating (\ref{Eq 1.1}) with respect to $x_k$ and $y_k$, we see that both $u$ and $v$ satisfy the
following system of partial differential equations of second order:
\begin{align}\label{Eq 1.2}
\frac{\partial^2}{\partial x_j \partial x_k}
+
\frac{\partial^2}{\partial y_j \partial y_k}
= 0
\quad \text{and} \quad
\frac{\partial^2}{\partial x_j \partial y_k}
-
\frac{\partial^2}{\partial x_k \partial y_j}
= 0
\qquad (j,k = 1,\dots,n).
\end{align}

For a complex variable $z_j = x_j + i y_j$, we define
\begin{align}\label{Eq 1.3}
\frac{\partial}{\partial z_j}
= \frac{1}{2}\left( \frac{\partial}{\partial x_j}
- i \frac{\partial}{\partial y_j} \right),\quad
\frac{\partial}{\partial \bar z_j}
= \frac{1}{2}\left( \frac{\partial}{\partial x_j}
+ i \frac{\partial}{\partial y_j} \right).
\end{align}
\begin{align*}
\partial =\sum\limits_{j=1}^n \frac{\partial }{\partial z_j} d z_j,\quad \ol{\partial} =\sum\limits_{j=1}^n \frac{\partial }{\partial \ol z_j} d \ol {z}_j \quad \text{and} \quad d=\partial+\ol{\partial}.
\end{align*}

We know that a function $f$ defined on an open subset $U\subset \mathbb{R}^n$ is said to be of $C^k$-class if $f$ is $k$-times continuously differentiable.

Let $f(z) = u(x,y) + i v(x,y)\;(x = (x_1,\dots,x_n), y = (y_1,\dots,y_n))$, where both $u$ and $v$ are of $C^2$-class. A direct calculation on (\ref{Eq 1.3}) shows that 
\begin{align}\label{Eq 1.5}
4\frac{\partial^2 f(z)}{\partial \bar z_j \partial z_k}=&\frac{\partial^2 u(x,y)}{\partial x_j x_k}+\frac{\partial^2 u(x,y)}{\partial y_j y_k}+i\left(\frac{\partial^2 v(x,y)}{\partial x_j x_k}+\frac{\partial^2 v(x,y)}{\partial y_j y_k}\right)\\&-i\left(\frac{\partial^2 u(x,y)}{\partial x_j y_k}-\frac{\partial^2 u(x,y)}{\partial x_k y_j}\right)+\left(\frac{\partial^2 v(x,y)}{\partial x_j y_k}-\frac{\partial^2 v(x,y)}{\partial x_k y_j}\right).\nonumber
\end{align}

\begin{defi} A function $f:\Omega \to \mathbb{C}$ defined on an open set $\Omega\subset \mathbb{C}^n$ is said to be holomorphic if $f$ is of $C^1$-class and satisfies
\begin{align}\label{Eq 1.4}
\ol{\partial} f=0,\;\; \text{i.e.},\;\; \frac{\partial f(z)}{\partial \bar z_j}=0\quad \text{on}\;\;\Omega\;\;\text{ for all j}.
\end{align}
\end{defi}

\subsection*{{\bf Pluriharmonic mapping}}

\begin{defi}
A real-valued function $\phi(x,y)$, where $x = (x_1,\dots,x_n)$ and $y = (y_1,\dots,y_n)$ is \emph{pluriharmonic} if it satisfies the conditions  \emph{(\ref{Eq 1.2})}. Thus a continuous complex-valued function $f(z)=u(x,y)+iv(x,y)$, where
$x = (x_1,\dots,x_n)$ and $y = (y_1,\dots,y_n)$ is a complex-valued \emph{pluriharmonic} function in a domain $\Omega\subset \mathbb{C}^n$, if both $u(x,y)$ and $v(x,y)$ are real-valued \emph{pluriharmonic} functions in $\Omega$.
If $u(x,y)$ and $v(x,y)$, where $x = (x_1,\dots,x_n)$ and $y = (y_1,\dots,y_n)$ satisfy \emph{(\ref{Eq 1.1})}, then we call $v(x,y)$, a \emph{pluriharmonic conjugate} of $u(x,y)$.
\end{defi}
\par Thus for functions $f(z)=u(x,y)+iv(x,y)$, where
$x = (x_1,\dots,x_n)$ and $y = (y_1,\dots,y_n)$ with continuous second order partial derivatives, it is clear from (\ref{Eq 1.5}) and (\ref{Eq 1.4}) that $\frac{\partial f(z)}{\partial z_j}$ is holomorphic on $\Omega$ for all $j$ if  $f(z)$ is pluriharmonic function.
\vspace{1.5mm}
\par In a simply connected domain $\Omega\subset \mathbb{C}^n$, let $f(z)$ be a complex-valued pluriharmonic function. We recall that $\frac{\partial f(z)}{\partial z_j}$ is holomorphic on $\Omega$ for all $j=1,2,\ldots,n$ if $f(z)$ is pluriharmonic and let $\frac{\partial h(z)}{\partial z_j}=\frac{\partial f(z)}{\partial z_j}\;(j=1,2,\ldots,n)$, where $h(z)$ is holomorphic in $\Omega$. Now let $g(z)=\ol {f(z)}-\ol {h(z)}$ and we observe that
\[\frac{\partial g(z)}{\partial \ol{z_j}}=\ol{\frac{\partial f(z)}{\partial z_j}}-\ol{\frac{\partial h(z)}{\partial z_j}}=0\quad \text{in}\;\Omega,\quad j=1,2,\ldots,n\]
by the definition of $h$. Thus $g(z)$ is holomorphic in $\Omega$. Therefore the pluriharmonic function $f(z)$ has the representation $f(z)=h(z)+\overline{g(z)}$, where $h(z)$ and $g(z)$ are holomorphic in $\Omega$. \vspace{1.2mm}

In one complex variable, harmonic functions are closely related to analytic functions since every harmonic function can locally be represented as the real part of an analytic function. When moving to higher dimensions, ordinary harmonicity is no longer sufficient to capture the complex structure of $\mathbb{C}^n$. This leads to the concept of pluriharmonic functions, which are functions whose restrictions to every complex line are harmonic. Thus, pluriharmonicity preserves the intrinsic geometry of complex spaces and provides a suitable framework for extending many classical results of geometric function theory to higher dimensions.

\smallskip
From the perspective of geometric function theory, pluriharmonic functions are important because they provide a broader class of mappings that retain many geometric properties of holomorphic functions. In particular, mappings of the form $f(z)=h(z)+\ol{g(z)}$,
where $h$ and $g$ are holomorphic functions in several complex variables, are pluriharmonic. Such mappings have become a major area of research due to their applications in studying univalence, distortion estimates, covering theorems, Schwarz-type lemmas, and Landau-Bloch-type results in higher dimensions.

\smallskip
Pluriharmonic functions are also closely connected with pseudoconvex domains, which are fundamental objects in several complex variables. Since every pluriharmonic function is simultaneously plurisubharmonic and plurisupeharmonic, it represents the \emph{flat} case of complex potential theory. This relationship enables the use of geometric methods to investigate the structure of complex domains and the behavior of holomorphic and harmonic mappings.

\smallskip
Therefore, pluriharmonic functions occupy a central position in modern geometric function theory. They generalize harmonic functions to several complex variables, connect analytic and geometric aspects of complex mappings, and provide a powerful framework for extending classical geometric results from the complex plane to higher-dimensional complex spaces. Their study continues to be an active research area, particularly in the theory of univalent mappings, invariant metrics, complex geometry, and multidimensional harmonic analysis.
\vspace{1.2mm}

Let $f(z)$ be a holomorphic function in a domain $\Omega\subset \mathbb{C}^n$. Then in the polydisk $\mathbb{P}\Delta(0;r)\subset \Omega$, $f(z)$ has a power series expansion in $z_1,\ldots,z_n$,
\begin{align*}
f(z)=\sum\limits_{\alpha_1,\alpha_2,\ldots,\alpha_n=0}^{\infty} a_{\alpha_1,\alpha_2,\ldots,\alpha_n}z_1^{\alpha_1}z_2^{\alpha_2}\ldots z_n^{\alpha_n}=
\sum\limits_{|\alpha|=0}^{\infty} a_{\alpha}z^{\alpha}=\sum\limits_{|\alpha|=0}^{\infty} P_{|\alpha|}(z),
\end{align*}
which is absolutely convergent in $\mathbb{P}\Delta(0;r)$, where the term $P_k(z)$ is a homogeneous polynomial of degree $k$. \vspace{1.2mm}

Let $G\not=\varnothing$ be an open subset of $\mathbb{C}^n$. Let $f$ be a holomorphic function on $G$. For a point $a\in\mathbb{C}^n$, we write 
\begin{align*}
	f(z)=\sum_{i=0}^{\infty}P_i(z-a),
\end{align*}
where the term $P_i(z-a)$ is either identically zero or a homogeneous polynomial of degree $i$. Denote the zero-multiplicity of $f$ at $a$ by 
\begin{align*}
	k=\min\{i:P_i(z-a)\not\equiv 0\}.
\end{align*}
Clearly $1$ is the zero-multiplicity of $f$ at $a$ when $f(a)=0$ and $\frac{\partial f(a)}{\partial z_j}\neq 0$ for some $j=1,2,\ldots,n$.

\subsection{\bf{Different class of pluriharmonic mappings}}
Let $\mathcal{H}_n$ denote the class of complex-valued pluriharmonic functions $f$ in $\mathbb{P} \Delta(0;1)$, normalized by 
\begin{align*}
	f(0)= 0\;\; \mbox{and}\;\;\nabla f(0):=\left(\frac{\partial f(0)}{\partial z_1},\ldots,\frac{\partial f(0)}{\partial z_n}\right)=(1,1,\ldots,1).
\end{align*}

Each function $f$ in $\mathcal{H}_n$ can be expressed as $f=h+\ol g$, where $h$ and $g$ are holomorphic functions in $\mathbb{P} \Delta(0;1)$. Here $h$ and $g$ are called the holomorphic and co-holomorphic parts of $f$ respectively, and have power series representations
\begin{align*}
	h(z)=\sum\limits_{j=1}^n z_j+\sum\limits_{k=2}^{\infty}\sum\limits_{|\alpha|=k} a_{\alpha} z^{\alpha}\quad {and}\quad g(z)=\sum\limits_{k=1}^{\infty}\sum\limits_{|\alpha|=k} b_{\alpha} z^{\alpha}.
\end{align*}

Obviously when $\dim(\mathbb{C}^n)=1$, $\mathcal{H}(=\mathcal{H}_1)$ is the the class of complex-valued harmonic functions $f$ in the unit disk $\mathbb{D}$, normalized by $f(0)=0$, $\frac{\partial f(0)}{\partial z}=1$.

If $g(z)\equiv 0$, then $\mathcal{H}_n$ reduces to the class $\mathcal{A}_n$ of holomorphic functions in the unit polydisk $\mathbb{P} \Delta(0;1)$ with
\begin{align*}
	f(0)= 0 \quad \text{and} \quad \nabla f(0).= (1, 1, \ldots, 1).
\end{align*}

 Clearly when $\dim(\mathbb{C}^n)=1$, $\mathcal{A}(=\mathcal{A}_1)$ is the the class of  holomorphic functions $f$ in the unit disk $\mathbb{D}$, normalized by $f(0)=0$, $\frac{\partial f(0)}{\partial z}=1$.
Let
\begin{align*}
	\mathcal{H}^0_n=\left\lbrace f\in\mathcal{H}_n: \ol{\partial} f(0)=0\right\rbrace.
\end{align*}

Hence for any function $f=h+\ol g$ in $\mathcal{H}^0_n$, its holomorphic and co-holomorphic parts can be represented by
\begin{align}\label{Eq 1.7}
	h(z)=\sum\limits_{j=1}^n z_j+\sum\limits_{k=2}^{\infty}\sum\limits_{|\alpha|=k} a_{\alpha} z^{\alpha}\quad {and} \quad g(z)=\sum\limits_{k=2}^{\infty}\sum\limits_{|\alpha|=k} b_{\alpha} z^{\alpha}
\end{align}
respectively. \vspace{1.2mm}

A complex-valued pluriharmonic mapping $f\in \mathcal{H}_n$ is said to be starlike if $f(\mathbb{P} \Delta(0;1))$ is a starlike
domain with respect to the origin. We denote the class of pluriharmonic starlike functions in $\mathbb{P} \Delta(0;1)$ by $\mathcal{S}^*_{\mathcal{H}_n}$. A function $f$ in $\mathcal{H}_n$ is said to be convex if $f(\mathbb{P} \Delta(0;1))$ is convex. We denote the class of pluriharmonic convex mappings in $\mathbb{P} \Delta(0;1)$ by $\mathcal{K}_{\mathcal{H}_n}$. A function $f\in\mathcal{H}_n$ is said to be close-to-convex if $f(\mathbb{P} \Delta(0;1))$ is a close-to-convex domain. We denote the class of pluriharmonic close-to-convex mappings in $\mathbb{P} \Delta(0;1)$ by $\mathcal{C}_{\mathcal{H}_n}$. Finally, denote by ${\mathcal{S}^{*0}_{\mathcal{H}_n}}$, $\mathcal{K}_{\mathcal{H}_n}^0$ and $\mathcal{C}_{\mathcal{H}_n}^0$, the
subclasses of $\mathcal{S}^*_{\mathcal{H}_n}$, $\mathcal{K}_{\mathcal{H}_n}$ and $\mathcal{C}_{\mathcal{H}_n}$ with $\ol{\partial}f(0)=(0,0,\ldots,0)$ respectively.\vspace{1.2mm}

Let $\mathcal{S}_{\mathcal{H}}$ be the subclass of $\mathcal{H}$ consisting of univalent and sense preserving harmonic
mappings.  In 1984, Clunie and Sheil-Small \cite{Clunie-Sheil-Small-1984} investigated the class $\mathcal{S}_{\mathcal{H}}$, together with some geometric subclasses. Subsequently, the class $\mathcal{S}_{\mathcal{H}}$ and its subclasses have been extensively studied by several authors (see \cite{Aizenberg-Aytuna-Djakov-JMAA-2001}-\cite{Ghosh-Allu-2019}, \cite{Hernandez-Martin-2013}-\cite{Liu-Yang-2019},  \cite{Muhanna-CVEE-2010}, \cite{Ponnusamy-Allu-Vuorinen-2009}-\cite{Ponnusamy-Kaliraj-Starkov-2017}).
\vspace{1.2mm}

In 2013, Li and Ponnusamy \cite{Li-Ponnusamy-2013} investigated the properties of functions in $\mathcal{P}_{\mathcal{H}}^{0}$  given by
\begin{align*}
	\mathcal{P}_{\mathcal{H}}^{0}
	:= \left\{
	f=h+\overline{g} \in \mathcal{H} :
	\Re\bigl(h'(z)\bigr) > |g'(z)|,\; z\in\mathbb{D}
	\right\}.
\end{align*}

The class $\mathcal{P}_{\mathcal{H}}^{0}$ is closely related to the class
\begin{align*}
	\mathcal{R}
	:= \left\{
	f\in\mathcal{S} :
	\Re\bigl(f'(z)\bigr)>0,\; z\in\mathbb{D}
	\right\},
\end{align*}
introduced by MacGregor \cite{MacGregor-1962}. It is known that a harmonic function $f=h+\ol g$ belongs to
$\mathcal{P}_{\mathcal{H}}^{0}$ if and only if the analytic functions
$h+\lambda g$ belong to $\mathcal{R}$ for each $\lambda$
($|\lambda|=1$) (see \cite{Li-Ponnusamy-2013,Li-Ponnusamy-2013a}).
Using this property, Li and Ponnusamy \cite{Li-Ponnusamy-2013} obtained coefficient bounds and the radius of convexity for functions in $\mathcal{P}_{\mathcal{H}}^{0}$.\vspace{1.2mm}

In 1977, Chichra \cite{Chichra-1977} studied the following class
\begin{align*}
	\mathcal{W}(\alpha)
	=
	\left\{
	f \in \mathcal{A} :
	\Re\!\left(f'(z)+\alpha z f''(z)\right)>0,
	\quad z\in\mathbb{D},\ \alpha\geq 0
	\right\}.
\end{align*}

In 2010, the regions of variability for functions in $\mathcal{W}(\alpha)$ were studied by Ponnusamy and Vasudevarao \cite{Ponnusamy-Vasudevarao-2010}, and Singh and Singh \cite{Singh-Singh-1982} established that $\mathcal{W}(1)$ is a subclass of $\mathcal{S}^{*}$.\vspace{1.2mm}

In 2014, Nagpal and Ravichandran \cite{Nagpal-Ravichandran-2014} studied the following class
\begin{align*}
	\mathcal{W}_{\mathcal{H}}^{0}
	:=
	\left\{
	f=h+\overline{g} \in \mathcal{H} :
	\Re\!\left(h'(z)+z h''(z)\right)
	>
	\left|g'(z)+z g''(z)\right|,
	\quad z\in\mathbb{D}
	\right\},
\end{align*}
which is the harmonic analogue of $\mathcal{W}(1)$.

It is known that $\mathcal{W}_{\mathcal{H}}^{0}$ is a subclass of both
$\mathcal{S}_{\mathcal{H}}^{*0}$ and $\mathcal{P}_{\mathcal{H}}^{0}$.
In particular, members of $\mathcal{W}_{\mathcal{H}}^{0}$ are fully starlike in $\mathbb{D}$.
Sharp coefficient bounds and growth theorems for functions in
$\mathcal{W}_{\mathcal{H}}^{0}$ have been obtained in \cite{Nagpal-Ravichandran-2014}.

\section{\bf {Preliminary results}}
In $2020$, Ghosh and Vasudevarao \cite{Ghosh-Allu-2020} have introduced the subclass $\mathcal{P}_{\mathcal{H}}^0(M)$ of $\mathcal{H}$ as follows.\vspace{1.2mm}

For $M>0$, let the class $\mathcal{P}^0_{\mathcal{H}}(M)$ be defined by
\begin{align*}
\mathcal{P}^0_{\mathcal{H}}(M)=\left\lbrace f=h+\ol g\in\mathcal{H}:\Re(zh''(z))>-M+|zg''(z)|\;\;\text{and}\;\;g'(0)=0 \right\rbrace.
\end{align*}

In \cite{Ghosh-Allu-2020}, Ghosh and Vasudevarao \cite{Ghosh-Allu-2020} proved that the class $\mathcal{P}_{\mathcal{H}}^0(M)$ is closely related to the class
\begin{align*}
	\mathcal{P}(M)=\left\lbrace \phi\in\mathcal{A}: \Re\left(z \phi''(z)\right)> -M\;\;\text{for}\;\;z\in \mathbb{D}\right\rbrace.
\end{align*}
	
The class $\mathcal{P}(M)$ has been studied extensively by Mocanu Mocanu \cite{Mocanu-1992} and Ponnusamy and Singh \cite{Ponnusamy-Singh-2002}.\vspace{1.2mm}

In the same paper, Ghosh and Vasudevarao \cite{Ghosh-Allu-2020} first established a result providing a one-to-one correspondence between the classes $\mathcal{P}_{\mathcal{H}}^0(M)$ and $\mathcal{P}(M)$.

\begin{theoA}\emph{\cite[Theorem 2.1]{Ghosh-Allu-2020}} A harmonic mapping $f=h+\ol g$ belongs to $\mathcal{P}_{\mathcal{H}}^0(M)$ if and only if the function $F_{\varepsilon}=h+\varepsilon g$ belongs to $\mathcal{P}(M)$ for each $\varepsilon$ ($|\varepsilon|=1$).
\end{theoA}

The following results due to Ghosh and Vasudevarao \cite{Ghosh-Allu-2020} provides sharp coefficient bounds for functions in $\mathcal{P}_{\mathcal{H}}^0(M)$.

\begin{theoB}\emph{\cite[Theorem 2.2]{Ghosh-Allu-2020}} Let $f=h+\ol g\in \mathcal{P}_{\mathcal{H}}^0(M)$ for $M>0$ be of the form
\begin{align}\label{THD}
h(z)=z+\sum\limits_{k=2}^{\infty} a_{k}z^k\quad \text{and}\quad g(z)=\sum\limits_{k=2}^{\infty}b_k z^k.
\end{align}
Then for $k\geq 2$,
\begin{align*}
|b_k|\leq \dfrac{2M}{k(k-1)}.
\end{align*}
The result is sharp for the function $f$ given by $f(z)=z-M\ol{z^k}/k(k-1)$.
\end{theoB}

\begin{theoC}\emph{\cite[Theorem 2.3]{Ghosh-Allu-2020}} Let $f=h+\ol g\in \mathcal{P}_{\mathcal{H}}^0(M)$ for $M>0$ be of the form \emph{ (\ref{THD})}. 
Then for $k\geq 2$,
\begin{enumerate}
\item[\emph{(i)}] $|a_k|+|b_k|\leq \dfrac{2M}{k(k-1)}$,\vspace{2mm}
\item[\emph{(ii)}] $\left||a_k|-|b_k|\right|\leq \dfrac{2M}{k(k-1)}$,\vspace{2mm}
\item[\emph{(iii)}] $|a_k|\leq \dfrac{2M}{k(k-1)}$.
\end{enumerate}
The result is sharp for the function $f$ given by $f'(z)=1-2M\log (1-z)$.
\end{theoC}

In the same paper, Ghosh and Vasudevarao \cite{Ghosh-Allu-2020} provides the growth estimate for functions in $\mathcal{P}_{\mathcal{H}}^0(M)$.

\begin{theoD}\emph{\cite[Theorem 2.4]{Ghosh-Allu-2020}} Let $f=h+\ol g\in \mathcal{P}_{\mathcal{H}}^0(M)$ for $M>0$ be of the form \emph{(\ref{THD})}. Then
\begin{align*}
|z|-2M\sum\limits_{k=2}^{\infty}\frac{|z|^k}{k(k-1)}\leq |f(z)|\leq |z|+2M\sum\limits_{k=2}^{\infty}\frac{|z|^k}{k(k-1)}.
\end{align*}

The right-hand inequality is sharp for the function $f$ given by $f'(z)=1-2M\log (1-z)$.
\end{theoD}

\section{{\bf Main Results}}\label{Sec-2}
We now introduce the class $\mathcal{P}_{\mathcal{H}_n^0}(M)$ in $\mathbb{P} \Delta(0_n;1_n)$ as follows.

\begin{defi} For $M>0$ and $z\in \mathbb{P} \Delta(0;1)$, let
	\begin{align*}
		&\mathcal{P}_{\mathcal{H}_n^0}(M)\\=&\left\lbrace f=h+\ol g\in \mathcal{H}_n^0: \sum\limits_{l=1}^n\sum\limits_{j=1}^n\sum\limits_{k=1}^n\Re\left(z_l\frac{\partial^2 h(z)}{\partial z_j\partial z_k}\right)>-M+\sum\limits_{l=1}^n\sum\limits_{j=1}^n\sum\limits_{k=1}^n\left|z_l\frac{\partial^2 g(z)}{\partial z_j\partial z_k}\right| \right\rbrace.
	\end{align*}
\end{defi}

When $\dim(\mathbb{C}^n)=1$, we denote $\mathcal{P}_{\mathcal{H}_1^0}(M)$ by $\mathcal{P}_{\mathcal{H}}^0(M)$.\vspace{1.2mm}

We will show that the class $\mathcal{P}_{\mathcal{H}_n^0}(M)$ is closely related to the class
\begin{align*}
	\mathcal{P}_n(M)=\left\lbrace \phi\in\mathcal{A}_n: \sum\limits_{l=1}^n\sum\limits_{j=1}^n\sum\limits_{k=1}^n\Re\left(z_l\frac{\partial^2 \phi(z)}{\partial z_j\partial z_k}\right)> -M\;\;\text{for}\;\;z\in \mathbb{P} \Delta(0_n;1_n)\right\rbrace.
\end{align*}

\smallskip
The novelty of this research lies in the extension of the theory of \emph{harmonic} functions in $\mathbb{C}$ to \emph{pluriharmonic} functions in higher-dimensional complex spaces ($\mathbb{C}^n$). In this paper, we study the analytic and geometric properties for the class $\mathcal{P}_{\mathcal{H}_n^0}(M)$. Here we also establish that the subclass $\mathcal{P}_{\mathcal{H}_n^0}(M)$ is not only the generalizations of holomorphic functions but also they are closely related to the holomorphic subclass $\mathcal{P}_n(M)$.
The sharp coefficient bounds and growth estimates provided for the class $\mathcal{P}_{\mathcal{H}_n^0}(M)$ represent a major advancement in the study of geometric function theory for several complex variables.

The following result provides a one-to-one correspondence between the classes $\mathcal{P}_{\mathcal{H}_n^0}(M)$ and $\mathcal{P}_n(M)$.

\begin{theo}\label{Th-1.3} A pluriharmonic mapping $f=h+\ol g$ is in $\mathcal{P}_{\mathcal{H}_n^0}(M)$ if and only if the function  $F_{\varepsilon}=h+\varepsilon g$ is in $\mathcal{P}_n(M)$ for each $\varepsilon\; (|\varepsilon|=1)$.
\end{theo}
\begin{proof} Let $f\in \mathcal{P}_{\mathcal{H}_n^0}(M)$. Then for each $\varepsilon\;(|\varepsilon|=1)$, we have
\begin{align*}
\sum\limits_{l=1}^n\sum\limits_{j=1}^n\sum\limits_{k=1}^n\Re\left(z_l\frac{\partial^2 F_{\varepsilon}(z)}{\partial z_j\partial z_k}\right)
=&\sum\limits_{l=1}^n\sum\limits_{j=1}^n\sum\limits_{k=1}^n\Re\left(z_l\frac{\partial^2 h(z)}{\partial z_j\partial z_k}+\varepsilon z_l\frac{\partial^2 g(z)}{\partial z_j\partial z_k}\right)\\>&
\sum\limits_{l=1}^n\sum\limits_{j=1}^n\sum\limits_{k=1}^n\Re\left(z_l\frac{\partial^2 h(z)}{\partial z_j\partial z_k}\right)-\sum\limits_{l=1}^n\sum\limits_{j=1}^n\sum\limits_{k=1}^n\left|z_l\frac{\partial^2 g(z)}{\partial z_j\partial z_k}\right|\\>&-M
\end{align*}
for all $z\in \mathbb{P} \Delta(0_n;1_n)$ and so $F_{\varepsilon}\in \mathcal{P}_n(M)$. 

Conversely, let $F_{\varepsilon}=h+\varepsilon g\in \mathcal{P}_n(M)$ for each $\varepsilon\;(|\varepsilon|=1)$. Then we have
\begin{align*}
\sum\limits_{l=1}^n\sum\limits_{j=1}^n\sum\limits_{k=1}^n\Re\left(z_l\frac{\partial^2 F_{\varepsilon}(z)}{\partial z_j\partial z_k}\right)
=&\sum\limits_{l=1}^n\sum\limits_{j=1}^n\sum\limits_{k=1}^n\left(\Re\left(z_l\frac{\partial^2 h(z)}{\partial z_j\partial z_k}\right)+\Re\left(\varepsilon z_l\frac{\partial^2 g(z)}{\partial z_j\partial z_k}\right)\right)>-M
\end{align*}
for all $z\in \mathbb{P} \Delta(0_n;1_n)$. 

Replacing $\varepsilon$ by $-\varepsilon$ in the above inequality, we have
\begin{align}\label{Th3:1.1}
\sum\limits_{l=1}^n\sum\limits_{j=1}^n\sum\limits_{k=1}^n\left(\Re\left(z_l\frac{\partial^2 h(z)}{\partial z_j\partial z_k}\right)-\Re\left(\varepsilon z_l\frac{\partial^2 g(z)}{\partial z_j\partial z_k}\right)\right)>-M
\end{align}
for all $z\in \mathbb{P} \Delta(0_n;1_n)$. Since $\varepsilon$ is arbitrary, we may choose it so that 
\begin{align}\label{Th3:1.2}
\Re\left(\varepsilon z_l\frac{\partial^2 g(z)}{\partial z_j\partial z_k}\right)=\left|z_l\frac{\partial^2 g(z)}{\partial z_j\partial z_k}\right|
\end{align}
for all $z\in \mathbb{P} \Delta(0_n;1_n)$ and for all $j,k,l\in\{1,2,\ldots,n\}$.
For the verification of the equality (\ref{Th3:1.2}), we substitute
\begin{align*}
\varepsilon=\frac{\ol{z_l\frac{\partial^2 g(z)}{\partial z_j\partial z_k}}}{\left|z_l\frac{\partial^2 g(z)}{\partial z_j\partial z_k}\right|}
\end{align*}
when $z_l\frac{\partial^2 g(z)}{\partial z_j\partial z_k}\neq 0$. Obviously (\ref{Th3:1.2}) is true when $z_l\frac{\partial^2 g(z)}{\partial z_j\partial z_k}= 0$. Now using (\ref{Th3:1.2}) to (\ref{Th3:1.1}), we get
\begin{align*}
\sum\limits_{l=1}^n\sum\limits_{j=1}^n\sum\limits_{k=1}^n\left(\Re\left(z_l\frac{\partial^2 h(z)}{\partial z_j\partial z_k}\right)-\left|z_l\frac{\partial^2 g(z)}{\partial z_j\partial z_k}\right|\right)>-M
\end{align*}
for all $z\in \mathbb{P} \Delta(0_n;1_n)$, i.e.,
\begin{align*}
\sum\limits_{l=1}^n\sum\limits_{j=1}^n\sum\limits_{k=1}^n\Re\left(z_l\frac{\partial^2 h(z)}{\partial z_j\partial z_k}\right)>-M+\sum\limits_{l=1}^n\sum\limits_{j=1}^n\sum\limits_{k=1}^n\left|z_l\frac{\partial^2 g(z)}{\partial z_j\partial z_k}\right|
\end{align*}
for all $z\in \mathbb{P} \Delta(0_n;1_n)$. Consequently $f\in \mathcal{P}_{\mathcal{H}_n^0}(M)$.
\end{proof}

The following result provides the sharp coefficient bounds for functions in $\mathcal{P}_{\mathcal{H}_n^0}(M)$.

\begin{theo}\label{Th-1.5} Let $f=h+\ol g\in \mathcal{P}_{\mathcal{H}_n^0}(M)$ and be given by \emph{(\ref{Eq 1.7})}. Then for any multi-index $\alpha=(\alpha_1,\alpha_2,\ldots,\alpha_n)$ such that $|\alpha|=m\geq 2$, we have
\begin{align*}
\displaystyle \sum\limits_{|\alpha|=m}|b_{\alpha}|\leq \dfrac{\binom{m+n-1}{n-1}M}{nm(m-1)}.
\end{align*}
The inequality is sharp.
\end{theo}
\begin{proof} Let $f=h+\ol g\in \mathcal{P}_{\mathcal{H}_n^0}(M)$. Then
\begin{align}\label{Th5:1.1}
\sum\limits_{l=1}^n\sum\limits_{j=1}^n\sum\limits_{k=1}^n\left|z_l\frac{\partial^2 g(z)}{\partial z_j\partial z_k}\right|
< M+\sum\limits_{l=1}^n\sum\limits_{j=1}^n\sum\limits_{k=1}^n\Re\left(z_l\frac{\partial^2 h(z)}{\partial z_j\partial z_k}\right)
\end{align}
for all $z=(z_1,z_2,\ldots,z_n)\in \mathbb{P} \Delta(0_n;1_n)$. It follows from (\ref{Eq 1.7}) that
\begin{align}\label{Th5:1.2}
\displaystyle h(z)=\sum\limits_{j=1}^n z_j+\sum\limits_{m=2}^{\infty}P_m(z)\quad {and} \quad g(z)=\sum\limits_{m=2}^{\infty} Q_m(z)
\end{align}
for all $z\in \mathbb{P} \Delta(0_n;1_n)$, where
\begin{align}\label{Th5:1.3}
\displaystyle P_{m}(z)=\sum\limits_{|\alpha|=m} a_{\alpha} z^{\alpha}\quad \text{and}\quad Q_m(z)=\sum\limits_{|\alpha|=m} b_{\alpha} z^{\alpha}
\end{align}
are homogeneous polynomials of degree $m\geq 2$ in $z\in \mathbb{P} \Delta(0_n;1_n)$. We know that the number of terms in $\sum\limits_{|\alpha|=m}$ is $\binom{|\alpha|+n-1}{n-1}$. A simple computation on (\ref{Th5:1.3}) shows that
\begin{align}\label{Th5:1.4}
\displaystyle z_l\frac{\partial^2 Q_m(z)}{\partial z_j\partial z_k}=\sum\limits_{|\alpha|=m}\alpha_{jk} b_{\alpha}z_1^{\alpha_1}\ldots z_{j-1}^{\alpha_{j-1}}z_{j}^{\alpha_j^*}z_{j+1}^{\alpha_{j+1}}\ldots z_{k-1}^{\alpha_{k-1}}z_{k}^{\alpha_k^{**}}z_{k+1}^{\alpha_{k+1}}\ldots z_l^{\alpha_l+1}\ldots z_n^{\alpha_n}
\end{align}
is a homogeneous polynomial of degree $m-1$ in $z_1,z_2,\ldots,z_n$, where
\begin{align}\label{ll.1}
\alpha_{jk}=
\begin{cases}
\alpha_j(\alpha_j-1), & \text{if}\; j=k\;\text{and}\;\alpha_j=|\alpha|,\\[2ex]
\alpha_j(\alpha_j-1), & \text{if}\; j=k\;\text{and}\;2\leq \alpha_j<|\alpha|,\\[2ex]
\alpha_j\alpha_k,& \text{if}\; j\neq k
\end{cases}
\;\;,\quad 
\alpha^*_j=
\begin{cases}
\alpha_j-2,& \text{if}\; j=k,\\[2ex]
\alpha_j-1,& \text{if}\; j\neq k
\end{cases}
\end{align}
and
\begin{align}\label{ll.2}
\alpha^{**}_k=
\begin{cases}
\alpha_k,& \text{if}\; j=k,\\[2ex]
\alpha_k-1,& \text{if}\; j\neq k.
\end{cases}
\end{align}

Now from (\ref{Th5:1.2}) and (\ref{Th5:1.4}), we obtain
\begin{align}\label{Th5:1.5}
\displaystyle z_l\frac{\partial^2 g(z)}{\partial z_j\partial z_k}&=
\sum\limits_{m=2}^{\infty} z_l\frac{\partial^2 Q_m(z)}{\partial z_j\partial z_k}\\=&\sum\limits_{m=2}^{\infty}\sum\limits_{|\alpha|=m}\alpha_{jk} b_{\alpha}z_1^{\alpha_1}\ldots z_{j-1}^{\alpha_{j-1}}z_{j}^{\alpha_j^*}z_{j+1}^{\alpha_{j+1}}\ldots z_{k-1}^{\alpha_{k-1}}z_{k}^{\alpha_k^{**}}z_{k+1}^{\alpha_{k+1}}\ldots z_l^{\alpha_l+1}\ldots z_n^{\alpha_n}.\nonumber
\end{align}
By applying Cauchy's integral formula to $z_l \frac{\partial^2 g(z)}{\partial z_j \partial z_k}$, it follows from \eqref{Th5:1.5} that
\begin{align}\label{Th5:1.6}
\displaystyle &(2\pi i)^n \alpha_{jk}b_{\alpha}\\=&\int\limits_{|z_1|=r_1}\ldots \int\limits_{|z_n|=r_n}\frac{z_l \frac{\partial^2 g(z)}{\partial z_j\partial z_k}\;dz_1\;d z_2\ldots d z_n}{z_1^{\alpha_1+1}\ldots z_{j-1}^{\alpha_{j-1}+1}z_{j}^{\alpha_j^*+1}z_{j+1}^{\alpha_{j+1}+1}\ldots z_{k-1}^{\alpha_{k-1}+1}z_{k}^{\alpha_k^{**}+1}z_{k+1}^{\alpha_{k+1}+1}\ldots z_l^{\alpha_l+2}\ldots z_n^{\alpha_n+1}},\nonumber
\end{align}
where $0<r_j<1$ for $j=1,2,\ldots,n$. 
We set $z=(r_1e^{i\theta_1},r_2e^{i\theta_2},\ldots,r_ne^{i\theta_n})\;(0\leq \theta_j\leq 2\pi)$.

Now from (\ref{Th5:1.6}), we have
\begin{align}\label{Th5:1.7}
\displaystyle &(2\pi)^n\sum\limits_{|\alpha|=m}\alpha_{jk}|b_{\alpha}|\\\leq &\binom{|\alpha|+n-1}{n-1}\int\limits_{0}^{2\pi}\ldots \int\limits_0^{2\pi}\frac{\left|r_le^{i\theta_l} \frac{\partial^2 g(z)}{\partial z_j\partial z_k}\right|\;d\theta_1\;d \theta_2\ldots d \theta_n}{r_1^{\alpha_1}\ldots r_{j-1}^{\alpha_{j-1}}r_{j}^{\alpha_j^*}r_{j+1}^{\alpha_{j+1}}\ldots r_{k-1}^{\alpha_{k-1}}r_{k}^{\alpha_k^{**}}r_{k+1}^{\alpha_{k+1}}\ldots r_l^{\alpha_l+1}\ldots r_n^{\alpha_n}}.\nonumber
\end{align}

If we take $r_j=r$ for $j=1,2,\ldots,n$, then from (\ref{Th5:1.7}), we obtain
\begin{align}\label{Th5:1.8}
&(2\pi)^nr^{m-1} \sum\limits_{l=1}^n\sum\limits_{|\alpha|=m}\sum\limits_{j=1}^n\sum\limits_{k=1}^n \alpha_{jk}|b_{\alpha}|\\=&
(2\pi)^nr^{m-1} \sum\limits_{|\alpha|=m}\sum\limits_{l=1}^n\sum\limits_{j=1}^n\sum\limits_{k=1}^n \alpha_{jk}|b_{\alpha}|\nonumber\\ \leq &\binom{|\alpha|+n-1}{n-1}\int\limits_{0}^{2\pi}\ldots \int\limits_0^{2\pi}\sum\limits_{l=1}^n\sum\limits_{j=1}^n\sum\limits_{k=1}^n\left|re^{i\theta_l}\frac{\partial^2 g(z)}{\partial z_j\partial z_k}\right|\;d\theta_1\;d \theta_2\ldots d \theta_n.\nonumber
\end{align}
However, we observe that if $\alpha_j=|\alpha|=m$ for $j=1,2,\dots,n$, then
\begin{align*}
\sum\limits_{|\alpha|=m}\sum\limits_{j=1}^n\sum\limits_{k=1}^n\alpha_{jk}|b_{\alpha}|=\sum\limits_{|\alpha|=m}\sum\limits_{j=1}^n \alpha_j(\alpha_j-1)|b_{\alpha}|=m(m-1)\sum\limits_{\substack{j=1\\\alpha_j=m}}^n|b_{\alpha}|
\end{align*}
and if $\alpha_j<|\alpha|$ for $j=1,2,\ldots,n$, then we have
\begin{align*}
\sum\limits_{|\alpha|=m}\sum\limits_{j=1}^n\sum\limits_{k=1}^n\alpha_{jk}|b_{\alpha}|=&
\sum\limits_{|\alpha|=m}\left(\sum\limits_{j=1}^n \alpha_j(\alpha_j-1)+2\sum\limits_{\substack{j,k=1\\j\neq k}}^n \alpha_j\alpha_k\right)|b_{\alpha}|\\=&m(m-1)\sum\limits_{\substack{|(\alpha_1,\alpha_2,\ldots,\alpha_n)|=m\\\alpha_j<m}}|b_{\alpha}|.
\end{align*}
Thus, we see that
\begin{align}\label{Th5:1.8a}
\sum\limits_{|\alpha|=m}\sum\limits_{j=1}^n\sum\limits_{k=1}^n\alpha_{jk}|b_{\alpha}|=m(m-1)\sum\limits_{|\alpha|=m}|b_{\alpha}|.
\end{align}
For any multi-index $\nu=(\nu_1,\nu_2,\ldots,\nu_n)$, we have
\begin{align}\label{Th5:1.9}
\int\limits_{0}^{2\pi}\ldots \int\limits_0^{2\pi} (e^{i\theta_1})^{\nu_1}\ldots (e^{i\theta_n})^{\nu_n}d\theta_1 \ldots d\theta_n=
\begin{cases}
0,& \nu\neq (0,0,\ldots,0),\\[2ex]
(2\pi)^n,& \nu=(0,0,\ldots,0).
\end{cases}
\end{align}
Moreover, we know that
\begin{align}\label{Th5:1.9b}
&\sum\limits_{l=1}^n\sum\limits_{j=1}^n\sum\limits_{k=1}^nz_l\frac{\partial^2 h(z)}{\partial z_j\partial z_k}\\=&
\sum\limits_{m=2}^{\infty}\sum\limits_{|\alpha|=m}\sum\limits_{l=1}^n\sum\limits_{j=1}^n\sum\limits_{k=1}^n\alpha_{jk} a_{\alpha}z_1^{\alpha_1}\ldots z_{j-1}^{\alpha_{j-1}}z_{j}^{\alpha_j^*}z_{j+1}^{\alpha_{j+1}}\ldots z_{k-1}^{\alpha_{k-1}}z_{k}^{\alpha_k^{**}}z_{k+1}^{\alpha_{k+1}}\ldots z_l^{\alpha_l+1}\ldots z_n^{\alpha_n}.\nonumber
\end{align}
Let $z=(r_1e^{i\theta_1},r_2e^{i\theta_2},\ldots,r_ne^{i\theta_n})\;(0\leq \theta_j\leq 2\pi)$. Using (\ref{Th5:1.9}), we deduce from (\ref{Th5:1.9b})  that
\begin{align*}
\frac{1}{(2\pi)^n}\int\limits_{0}^{2\pi}\ldots \int\limits_0^{2\pi}\sum\limits_{l=1}^n\sum\limits_{j=1}^n\sum\limits_{k=1}^n r_le^{i\theta_l}\frac{\partial^2 h(z)}{\partial z_j\partial z_k}\;d\theta_1\;d \theta_2\ldots d \theta_n=0,
\end{align*}
i.e.,
\begin{align}\label{Th5:1.9a}
\frac{1}{(2\pi)^n}\int\limits_{0}^{2\pi}\ldots \int\limits_0^{2\pi}\sum\limits_{l=1}^n\sum\limits_{j=1}^n\sum\limits_{k=1}^n \Re\left(r_le^{i\theta_l}\frac{\partial^2 h(z)}{\partial z_j\partial z_k}\right)\;d\theta_1\;d \theta_2\ldots d \theta_n=0.
\end{align}

In view of (\ref{Th5:1.1}) and (\ref{Th5:1.9a}), we obtain from (\ref{Th5:1.8}) that
\begin{align*}
&r^{m-1} nm(m-1)\sum\limits_{|\alpha|=m}|b_{\alpha}|\\\leq &\binom{|\alpha|+n-1}{n-1}\frac{1}{(2\pi)^n}\int\limits_{0}^{2\pi}\ldots \int\limits_0^{2\pi}\left(M+\sum\limits_{l=1}^n\sum\limits_{j=1}^n\sum\limits_{k=1}^n\Re\left(r_le^{i\theta_l}\frac{\partial^2 h(z)}{\partial z_j\partial z_k}\right)\right)\;d\theta_1\;d \theta_2\ldots d \theta_n\nonumber\\\leq&
\binom{|\alpha|+n-1}{n-1}M+\frac{1}{(2\pi)^n}\int\limits_{0}^{2\pi}\ldots \int\limits_0^{2\pi}\sum\limits_{l=1}^n\sum\limits_{j=1}^n\sum\limits_{k=1}^n \Re\left(r_le^{i\theta_l}\frac{\partial^2 h(z)}{\partial z_j\partial z_k}\right)\;d\theta_1\;d \theta_2\ldots d \theta_n\\=&
\binom{|\alpha|+n-1}{n-1}M.
\end{align*}

Letting $r\to 1^-$, we obtain the desired result. \vspace{1.2mm}

To show that the bound is sharp, we consider the following function
\begin{align*}
f(z)=\sum\limits_{j=1}^n z_j+\sum\limits_{m=1}^{\infty}\sum\limits_{|\alpha|=m} b_{\alpha} \ol{z^{\alpha}},
\end{align*}
where 
\begin{align*}
b_{\alpha}=-\frac{M}{nm(m-1)}
\end{align*}
for all multi-index $\alpha=(\alpha_1,\alpha_2,\ldots,\alpha_n)$ such that $|\alpha|=m$.
It is easy to see that $f\in \mathcal{P}_{\mathcal{H}_n^0}(M)$, and 
\begin{align*}
\sum\limits_{|\alpha|=m}|b_{\alpha}(f)|=\frac{\binom{m+n-1}{n-1}M}{nm(m-1)}.
\end{align*} 
\end{proof}

\begin{lem}\label{Lm-3.1}\emph{\cite[Theorem 6.1.4]{Graham-Kohr}} Let $f$ be holomorphic in the polydisk $\mathbb{P}\Delta(0_n;1_n)$ such that $|f(z)|\leq 1$ for all $z\in \mathbb{P}\Delta(0_n;1_n)$. Then
\begin{align*}
\left|\frac{\partial^{|\alpha|} f(0)}{\partial z_1^{\alpha_1}\ldots \partial z_n^{\alpha_n}}\right|\leq \alpha!
\end{align*}
 for multi-index $\alpha=(\alpha_1,\ldots, \alpha_n)$.
\end{lem}

New we state the multidimensional version of Theorem C ($(i)$ and $(ii)$).
\begin{theo}\label{Th-1.6} Let $f=h+\ol g\in \mathcal{P}_{\mathcal{H}_n^0}(M)$ and  be given by \emph{(\ref{Eq 1.7})}. Then for any multi-index $\alpha=(\alpha_1,\alpha_2,\ldots,\alpha_n)$ such that $|\alpha|=m\geq 2$, we have
\begin{enumerate}
\item[\emph{(i)}] $\displaystyle \left|\sum\limits_{|\alpha|=m}a_{\alpha}\right|+\left|\sum\limits_{|\alpha|=m}b_{\alpha}\right|\leq \dfrac{2\binom{m+n-1}{n-1}M}{nm(m-1)}$,\vspace{1.2mm}
\item[\emph{(ii)}] $\displaystyle\left|\;\left|\sum\limits_{|\alpha|=m}a_{\alpha}\right|-\left|\sum\limits_{|\alpha|=m}b_{\alpha}\right|\;\right|\leq \dfrac{2\binom{m+n-1}{n-1}M}{nm(m-1)}$,\vspace{1.2mm}
\item[\emph{(iii)}] $\displaystyle\left|\sum\limits_{|\alpha|=m}a_{\alpha}\right|\leq \dfrac{2\binom{m+n-1}{n-1}M}{nm(m-1)}$.
\end{enumerate}
All three inequalities are sharp.
\end{theo}

\begin{proof} Let $f=h+\ol g\in \mathcal{P}_{\mathcal{H}_n^0}(M)$. Then from Theorem \ref{Th-1.3}, we see that the function $F_{\varepsilon}=h+\varepsilon g$ belongs to $\mathcal{P}_{n}(M)$ for each $\varepsilon\; (|\varepsilon|=1)$, and further
\begin{align}\label{Th6:1.1}
\sum\limits_{l=1}^n\sum\limits_{j=1}^n\sum\limits_{k=1}^n\Re\left(z_l\frac{\partial^2 F_{\varepsilon}(z)}{\partial z_j\partial z_k}\right)=\sum\limits_{l=1}^n\sum\limits_{j=1}^n\sum\limits_{k=1}^n\left(\Re\left(z_l\frac{\partial^2 h(z)}{\partial z_j\partial z_k}\right)+\Re\left(\varepsilon z_l\frac{\partial^2 g(z)}{\partial z_j\partial z_k}\right)\right)>-M
\end{align}
for all $z=(z_1,z_2,\ldots,z_n)\in \mathbb{P} \Delta(0_n;1_n)$. Let 
\begin{align}\label{Th6:1.2}
p(z)=\frac{\sum\limits_{l=1}^n\sum\limits_{j=1}^n\sum\limits_{k=1}^nz_l\frac{\partial^2 F_{\varepsilon}(z)}{\partial z_j\partial z_k}+M}{M}.
\end{align}
Clearly $p(z)$ is a holomorphic in $\mathbb{P}\Delta(0_n;1_n)$. Using (\ref{Th6:1.1}) to (\ref{Th6:1.2}), we see that $\Re\{p(z)\} > 0$ for all $z=(z_1,z_2,\ldots,z_n)\in \mathbb{P} \Delta(0_n;1_n)$. Also from \eqref{Th5:1.2} and \eqref{Th5:1.3}, we deduce that $p(0)=1$.
Therefore, we may assume that
\begin{align}\label{Th6:1.2a}
p(z)=1+\sum\limits_{m=1}^{\infty}\sum\limits_{|\alpha|=m}p_{\alpha}z^{\alpha}
\end{align}
for all $z\in \mathbb{P}\Delta(0_n;1_n)$.\vspace{1.2mm}

Let $\omega_l=e^{\frac{2\pi \iota}{k}l}$, where $k\geq 1$ is an integer. We see that $\sum_{l=1}^k \omega_l^s=0$ if $s\leq k-1$. 
Let
\begin{align*}
	q(z)=\frac{1}{k}\sum\limits_{l=1}^k p\left(\omega_l z\right).
\end{align*}

The function $q$ is clearly holomorphic in $\mathbb{P}\Delta(0;1_n)$ such that $\Re\{q(z)\} > 0$ for all $z\in \mathbb{P}\Delta(0_n;1_n)$. Furthermore, we see that $q(0)=1$ and $q$ has the series expansion
\begin{align*}
q(z)=1+\sum\limits_{m=1}^{\infty}\sum\limits_{|\alpha|=mk} p_{\alpha}z^{\alpha},
\end{align*}
for all $z\in \mathbb{P}\Delta(0_n;1_n)$. Let
\begin{align}
	\label{Eq-3.20}\phi(z)=\frac{q(z)-1}{q(z)+1},
\end{align}
for all $z\in \mathbb{P}\Delta(0_n;1_n)$. It is easy to verify that $|\phi(z)|<1$ for all $z\in \mathbb{P}\Delta(0_n;1_n)$ and $k$ is the zero-multiplicity of $\phi$ at $0$. We expand $\phi(z)$ to a Taylor series with multi-index $\alpha$, as follows
\begin{align}
	\label{Eq-3.21} \phi(z)=\sum\limits_{m=k}^{\infty}\sum\limits_{|\alpha|=m}q_{\alpha}z^{\alpha}.
\end{align}

 Now from \eqref{Eq-3.20} and \eqref{Eq-3.21}, we deduce that
\begin{align*}
	\sum\limits_{|\alpha|=k} p_{\alpha}z^{\alpha}+\sum\limits_{m=2}^{\infty}\sum\limits_{|\alpha|=mk} p_{\alpha}z^{\alpha}&=2 \sum\limits_{m=k}^{\infty}\sum\limits_{|\alpha|=m}q_{\alpha}z^{\alpha}\\&+
	\left(\sum\limits_{|\alpha|=k} p_{\alpha}z^{\alpha}+\sum\limits_{m=2}^{\infty}\sum\limits_{|\alpha|=mk} p_{\alpha}z^{\alpha}\right)\left(\sum\limits_{m=k}^{\infty}\sum\limits_{|\alpha|=m}q_{\alpha}z^{\alpha}\right).
\end{align*}
It then follows that
\begin{align}
	\label{Eq-3.23} \sum\limits_{|\alpha|=k} \left(p_{\alpha}-2q_{\alpha}\right)z^{\alpha}+\text{higher degree terms of}\; (z_1^{\alpha_1},\ldots,z_n^{\alpha_n})\equiv 0,
\end{align}
holds for all $z\in \mathbb{P}\Delta(0_n;1_n)$. Then by a simple calculation, we get from \eqref{Eq-3.23} that 
\begin{align}
	\label{Eq-3.24}\sum\limits_{|\alpha|=k} \left(p_{\alpha}-2q_{\alpha}\right)=0.
\end{align}

Using Lemma \ref{Lm-3.1} to \eqref{Eq-3.21}, we have $|q_{\alpha}|\leq 1$
for multi-index $\alpha$ and so from \eqref{Eq-3.24}, we conclude that
\begin{align}\label{Th6:1.3}
	\sum\limits_{|\alpha|=k} |p_{\alpha}|\leq 2\sum\limits_{|\alpha|=k} |q_{\alpha}|\leq 2\sum\limits_{|\alpha|=k} 1\leq 2\binom{k+n-1}{n-1},
\end{align}
where $\alpha=(\alpha_1,\alpha_2,\ldots,\alpha_n)$ such that $|\alpha|=k$.\vspace{1.2mm}


On the other hand, from (\ref{Th6:1.2}) and (\ref{Th6:1.2a}), we have
\begin{align}\label{Th6:1.4}
\sum\limits_{l=1}^n\sum\limits_{j=1}^n\sum\limits_{k=1}^nz_l\frac{\partial^2 F_{\varepsilon}(z)}{\partial z_j\partial z_k}=M\sum\limits_{m=1}^{\infty}\sum\limits_{|\alpha|=m}p_{\alpha}z^{\alpha}.
\end{align}

It is clear that
\begin{align*}
F_{\varepsilon}(z)=\sum\limits_{k=1}^n z_j+\sum\limits_{m=2}^{\infty}\left(P_m(z)+\varepsilon Q_m(z)\right),
\end{align*}
for all $z\in \mathbb{P} \Delta(0_n;1_n)$, where $P_m(z)$ and $Q_m(z)$ are defined in (\ref{Th5:1.3}). Clearly
\begin{align*}
z_l\frac{\partial^2 F_{\varepsilon}(z)}{\partial z_j\partial z_k}=&
\sum\limits_{m=2}^{\infty} z_l\frac{\partial^2 \left(P_m(z)+\varepsilon Q_m(z)\right)}{\partial z_j\partial z_k}\\=&\sum\limits_{m=2}^{\infty}\sum\limits_{|\alpha|=m}\alpha_{jk} c_{\alpha}z_1^{\alpha_1}\ldots z_{j-1}^{\alpha_{j-1}}z_{j}^{\alpha_j^*}z_{j+1}^{\alpha_{j+1}}\ldots z_{k-1}^{\alpha_{k-1}}z_{k}^{\alpha_k^{**}}z_{k+1}^{\alpha_{k+1}}\ldots z_l^{\alpha_l+1}\ldots z_n^{\alpha_n},\nonumber
\end{align*}
where $c_{\alpha}=a_{\alpha}+\varepsilon b_{\alpha}$ and $\alpha_{jk}$, $\alpha^*_j$ and $\alpha^{**}_k$ are defined as in (\ref{ll.1}) and (\ref{ll.2}) respectively. Consequently
\begin{align}\label{Th6:1.5}
&\sum\limits_{l=1}^n\sum\limits_{j=1}^n\sum\limits_{k=1}^nz_l\frac{\partial^2 F_{\varepsilon}(z)}{\partial z_j\partial z_k}\\=&
\sum\limits_{m=2}^{\infty}\sum\limits_{|\alpha|=m}\sum\limits_{l=1}^n\sum\limits_{j=1}^n\sum\limits_{k=1}^n\alpha_{jk} c_{\alpha}z_1^{\alpha_1}\ldots z_{j-1}^{\alpha_{j-1}}z_{j}^{\alpha_j^*}z_{j+1}^{\alpha_{j+1}}\ldots z_{k-1}^{\alpha_{k-1}}z_{k}^{\alpha_k^{**}}z_{k+1}^{\alpha_{k+1}}\ldots z_l^{\alpha_l+1}\ldots z_n^{\alpha_n}.\nonumber
\end{align}

Now from (\ref{Th6:1.4}) and (\ref{Th6:1.5}), we obtain
\begin{align}\label{Th6:1.5a}
&\sum\limits_{|\alpha|=m}\sum\limits_{l=1}^n\sum\limits_{j=1}^n\sum\limits_{k=1}^n\alpha_{jk} c_{\alpha}z_1^{\alpha_1}\ldots z_{j-1}^{\alpha_{j-1}}z_{j}^{\alpha_j^*}z_{j+1}^{\alpha_{j+1}}\ldots z_{k-1}^{\alpha_{k-1}}z_{k}^{\alpha_k^{**}}z_{k+1}^{\alpha_{k+1}}\ldots z_l^{\alpha_l+1}\ldots z_n^{\alpha_n}\\=&M\sum\limits_{|\alpha|=m}p_{\alpha}z^{\alpha},\nonumber
\end{align}
where $\alpha=(\alpha_1,\alpha_2,\ldots,\alpha_n)$ such that $|\alpha|=m\geq 2$. Putting $z=(r,r,,\ldots,r)$ on both sides of (\ref{Th6:1.5a}) and then canceling $r$ from the both sides, we get
\begin{align*}
\sum\limits_{|\alpha|=m}\sum\limits_{l=1}^n\sum\limits_{j=1}^n\sum\limits_{k=1}^n\alpha_{jk}c_{\alpha}=M\sum\limits_{|\alpha|=m} p_{\alpha},
\end{align*}
for all $\alpha=(\alpha_1,\alpha_2,\ldots,\alpha_n)$ such that $|\alpha|=m\geq 2$. Since $c_{\alpha}=a_{\alpha}+\varepsilon b_{\alpha}$, we have 
\begin{align}\label{Th6:1.5b}
\sum\limits_{|\alpha|=m}\sum\limits_{l=1}^n\sum\limits_{j=1}^n\sum\limits_{k=1}^n\alpha_{jk}a_{\alpha}+\varepsilon \sum\limits_{|\alpha|=m}\sum\limits_{l=1}^n\sum\limits_{j=1}^n\sum\limits_{k=1}^n\alpha_{jk}b_{\alpha} =M\sum\limits_{|\alpha|=m} p_{\alpha},
\end{align}	
for all $\alpha=(\alpha_1,\alpha_2,\ldots,\alpha_n)$ such that $|\alpha|=m\geq 2$. Now in view of (\ref{Th5:1.8a}), we conclude that
\begin{align}\label{Th6:1.9}
	\sum\limits_{|\alpha|=m}\sum\limits_{l=1}^n\sum\limits_{j=1}^n\sum\limits_{k=1}^n\alpha_{jk}a_{\alpha}=nm(m-1)\sum\limits_{|\alpha|=m}a_{\alpha}
\end{align}
and
\begin{align}\label{Th6:1.10}
	\sum\limits_{|\alpha|=m}\sum\limits_{l=1}^n\sum\limits_{j=1}^n\sum\limits_{k=1}^n\alpha_{jk}b_{\alpha}=nm(m-1)\sum\limits_{|\alpha|=m}b_{\alpha}.
\end{align}

Using \eqref{Th6:1.3}, \eqref{Th6:1.9} and \eqref{Th6:1.10} to \eqref{Th6:1.5b}, we find that
\begin{align}\label{Th6:1.6}
	\left|\sum\limits_{|\alpha|=m}a_{\alpha}+\varepsilon \sum\limits_{|\alpha|=m} b_{\alpha}\right|\leq 2M\frac{\binom{m+n-1}{n-1}}{nm(m-1)},
\end{align}
for all multi-index $\alpha=(\alpha_1,\alpha_2,\ldots,\alpha_n)$ such that $|\alpha|=m\geq 2$ and for all $\varepsilon$ such that $|\varepsilon|=1$. 

Suppose $A_{\alpha}=\sum\limits_{|\alpha|=m}a_{\alpha}\neq 0$ and $B_{\alpha}=\sum\limits_{|\alpha|=m}b_{\alpha}\neq 0$. If we choose 
\begin{align*}
\varepsilon=\frac{A_{\alpha}}{|A_{\alpha}|}\frac{\ol B_{\alpha}}{|B_{\alpha}|},
\end{align*}
then $|\varepsilon|=1$ and so from (\ref{Th6:1.6}), we have
\begin{align}\label{Th6:1.7}
|A_{\alpha}+\varepsilon B_{\alpha}|=|A_{\alpha}|+|B_{\alpha}|\leq 2M\frac{\binom{m+n-1}{n-1}}{nm(m-1)}
\end{align}
for multi-index $\alpha=(\alpha_1,\alpha_2,\ldots,\alpha_n)$ such that $|\alpha|=m\geq 2$. If either $A_{\alpha}=0$ or $B_{\alpha}=0$, then (\ref{Th6:1.7}) also holds. Therefore, we have
\begin{align}\label{Th6:1.7a}
\displaystyle	\left|\sum\limits_{|\alpha|=m}a_{\alpha}\right|+\left|\sum\limits_{|\alpha|=m}b_{\alpha}\right|\leq \dfrac{2\binom{m+n-1}{n-1}M}{nm(m-1)}
\end{align}		
for multi-index $\alpha=(\alpha_1,\alpha_2,\ldots,\alpha_n)$ such that $|\alpha|=m\geq 2$.\vspace{1,2mm}

On the other hand, from \eqref{Th6:1.6}, we have
\begin{align}\label{Th6:1.8}
\left||A_{\alpha}|-|B_{\alpha}|\right|\leq |A_{\alpha}+\varepsilon B_{\alpha}|\leq \dfrac{2\binom{m+n-1}{n-1}M}{nm(m-1)}
\end{align}
for multi-index $\alpha=(\alpha_1,\alpha_2,\ldots,\alpha_n)$ such that $|\alpha|=m\geq 2$. 
Therefore, we obtain
\begin{align*}
 \displaystyle \left|\;\left|\sum\limits_{|\alpha|=m}a_{\alpha}\right|-\left|\sum\limits_{|\alpha|=m}b_{\alpha}\right|\;\right|\leq \dfrac{2\binom{m+n-1}{n-1}M}{nm(m-1)}
\end{align*}
for multi-index $\alpha=(\alpha_1,\alpha_2,\ldots,\alpha_n)$ such that $|\alpha|=m\geq 2$. 

It is clear from \eqref{Th6:1.7a} that 
\begin{align*}
\left|\sum\limits_{|\alpha|=m}a_{\alpha}\right|\leq \dfrac{2\binom{m+n-1}{n-1}M}{nm(m-1)}
\end{align*}
for multi-index $\alpha=(\alpha_1,\alpha_2,\ldots,\alpha_n)$ such that $|\alpha|=m\geq 2$.\vspace{1.2mm}

To show that the inequalities are sharp, we consider the function $f_1(z)$ defined by
\begin{align}\label{Ex1}
f_1(z)=\sum\limits_{k=1}^n z_k+\sum\limits_{m=2}^{\infty}\sum\limits_{|\alpha|=m} a_{\alpha}z^{\alpha},
\end{align}
where
\begin{align*}
a_{\alpha}=\frac{2M}{nm(m-1)},
\end{align*}
 $\alpha=(\alpha_1,\alpha_2,\ldots,\alpha_n)$ such that $|\alpha|=m\geq 2$.
 It is easy to see that $f_1\in \mathcal{P}_{\mathcal{H}_n^0}(M)$. Also for $|\alpha|=m\geq 2$, we have
\begin{align*}
\left|\sum\limits_{|\alpha|=m}a_{\alpha}(f_1)\right|=\frac{2\binom{m+n-1}{n-1}M}{nm(m-1)}.
\end{align*} 

Hence all the three inequalities are sharp for the function given by \eqref{Ex1}.
\end{proof}

In the following result, we have established the growth estimates for the functions in the class $\mathcal{P}_{\mathcal{H}_n^0}(M)$.

\begin{theo}\label{Th-8}Let $f=h+\ol g\in \mathcal{P}_{\mathcal{H}_n^0}(M)$ and be given by (\ref{Eq 1.7}). Then
\begin{enumerate}
\item[\emph{(i)}] $\displaystyle |f(z)|\geq \left|\sum\limits_{k=1}^nz_k\right|-2M\sum\limits_{m=2}^{\infty} \binom{m+n-2}{n-1}\frac{\|z\|_{\infty}^{m}}{m(m-1)}$;\vspace{1.2mm}
\item[\emph{(ii)}] $\displaystyle |f(z)|\leq \left|\sum\limits_{k=1}^nz_k\right|+2M\sum\limits_{m=2}^{\infty} \binom{m+n-2}{n-1}\frac{\|z\|_{\infty}^{m}}{m(m-1)}$.
\end{enumerate}
All the inequalities are sharp.
\end{theo}

\begin{proof} Let $f=h+\ol g\in \mathcal{P}_{\mathcal{H}_n^0}(M)$. Then by Theorem \ref{Th-1.3}, we see that the function $F_{\varepsilon}=h+\varepsilon g$ belongs to $\mathcal{P}_{n}(M)$ for each $\varepsilon\; (|\varepsilon|=1)$, and furthermore  (\ref{Th6:1.1}) holds
for all $z=(z_1,z_2,\ldots,z_n)\in \mathbb{P} \Delta(0_n;1_n)$. Then there exists a holomorphic function $\hat P(z)$ in $\mathbb{P}\Delta(0_n;1_n)$ of the form
\begin{align}\label{Th8:1.1}
\hat P(z)=1+\sum\limits_{k=1}^{\infty}\hat P_k(z)
\end{align}
where 
\begin{align}\label{Th8:1.1a}
\hat P_k(z)=\sum\limits_{|\alpha|=k}p_{\alpha}z^{\alpha}
\end{align}
is a homogeneous polynomial of degree $k$ such that $\Re\{\hat P(z)\} > 0$ and 
\begin{align}\label{Th8:1.2}
\frac{\sum\limits_{l=1}^n\sum\limits_{j=1}^n\sum\limits_{k=1}^nz_l\frac{\partial^2 F_{\varepsilon}(z)}{\partial z_j\partial z_k}+M}{M}=\hat P(z)
\end{align}
in $\mathbb{P}\Delta(0_n;1_n)$. Using (\ref{Th8:1.1}) to (\ref{Th8:1.2}), we get
\begin{align}\label{Th8:1.3}
\sum\limits_{l=1}^n\sum\limits_{j=1}^n\sum\limits_{k=1}^nz_l\frac{\partial^2 F_{\varepsilon}(z)}{\partial z_j\partial z_k}=M\sum\limits_{k=1}^{\infty}\hat P_k(z)
\end{align}
for all $z\in \mathbb{P}\Delta(0_n;1_n)$. In view of (\ref{Th8:1.1a}) and using \eqref{Th6:1.3} to (\ref{Th8:1.1}), we obtain 
\begin{align}\label{Th8:1.4}
\sum\limits_{|\alpha|=m} |p_{\alpha}|\leq 2\binom{m+n-1}{n-1},
\end{align}
where $\alpha=(\alpha_1,\alpha_2,\ldots,\alpha_n)$ such that $|\alpha|=m\geq 1$. 

Let $z\in \mathbb{P}\Delta(0_n;1_n)$ be fixed in such a way that $z\neq 0$. Obviously $\|z\|_{\infty}<1$. 
We define 
\begin{align}\label{Th8:1.6}
	\tilde h(t)=h\left(\frac{z}{||z||_{\infty}}t\right)=\frac{1}{||z||_{\infty}}\left(\sum\limits_{k=1}^n z_k\right)t+\sum\limits_{m=2}^{\infty}P_m\left(\frac{z}{||z||_{\infty}}\right)t^m
\end{align}
and
\begin{align}\label{Th8:1.6a}
 \tilde g(t)=g\left(\frac{z}{||z||_{\infty}}t\right)=\sum\limits_{k=2}^{\infty} Q_m\left(\frac{z}{||z||_{\infty}}\right)t^m,
\end{align}
where $h(z)$ and $g(z)$ are defined by (\ref{Th5:1.2}) and $t\in\mathbb{C}$ such that $|t|<1$. For a fixed $z\in \mathbb{P}\Delta(0_n;1_n)$, we can say that $\tilde g(t)$ and $\tilde h(t)$ are analytic in $|t|<1$. Obviously for a fixed $z\in \mathbb{P}\Delta(0_n;1_n)$, the following function
\begin{align}\label{Th8:1.7}
\tilde F_{\varepsilon}(t)=F_{\varepsilon}\left(\frac{z}{||z||_{\infty}}t\right)=\tilde h(t)+\varepsilon \tilde g(t)
\end{align}
is analytic in $|t|<1$, where $\tilde h(t)$ and $\tilde g(t)$ are defined by (\ref{Th8:1.6}) and \eqref{Th8:1.6a} respectively. Again from (\ref{Th8:1.6}) and (\ref{Th8:1.6a}), we see that
\begin{align}\label{Th8:1.15}
	\tilde f(t)=f\left(\frac{z}{||z||_{\infty}}t\right)=\tilde h(t)+\ol{\tilde g(t)}.
\end{align}

In term of one variable $t$, we deduce from (\ref{Th8:1.3}) that 
\begin{align}\label{Th8:1.8}
t\frac{d^2 \tilde F_{\varepsilon}(t)}{d t^2}=M\sum\limits_{k=1}^{\infty}\hat P_{k}\left(\frac{z}{||z||_{\infty}}\right)t^k.
\end{align}

Applying (\ref{Th8:1.4}) to (\ref{Th8:1.1a}), we get
\begin{align}\label{Th8:1.5}
\left|\hat P_k\left(\frac{z}{||z||_{\infty}}\right)\right|\leq \sum\limits_{|\alpha|=k} |p_{\alpha}|\leq 2\binom{k+n-1}{n-1},
\end{align}
for all $z\in \mathbb{P}\Delta(0;1)$, where $\alpha=(\alpha_1,\alpha_2,\ldots,\alpha_n)$ such that $|\alpha|=k\geq 1$.
On the other hand, using (\ref{Th8:1.6}) and (\ref{Th8:1.6a}) to (\ref{Th8:1.7}), we get
\begin{align*}
\frac{d \tilde F_{\varepsilon}(0)}{d t}=\frac{z_1+z_2+\ldots+z_n}{||z||_{\infty}}.
\end{align*}

Therefore
\begin{align}\label{Th8:1.9}
\frac{d \tilde F_{\varepsilon}(t)}{d t}-\frac{z_1+z_2+\ldots+z_n}{||z||_{\infty}}=\frac{d \tilde F_{\varepsilon}(t)}{d t}-\frac{d \tilde F_{\varepsilon}(0)}{d t}=\int\limits_{0}^{t} \frac{d^2 \tilde F_{\varepsilon}(s)}{d s^2} ds.
\end{align}

In view of (\ref{Th8:1.8}) and (\ref{Th8:1.5}), we obtain from (\ref{Th8:1.9}) that
\begin{align}\label{Th8:1.10}
\left|\;\left|\frac{d \tilde F_{\varepsilon}(t)}{d t}\right|-\left|\sum\limits_{k=1}^nz_k/||z||_{\infty}\right|\;\right|\leq&\left|\frac{d \tilde F_{\varepsilon}(t)}{d t}-\sum\limits_{k=1}^n z_k/||z||_{\infty}\right|\\=&\left|\int\limits_{0}^{t}\frac{d^2 F_{\varepsilon}(s)}{d s^2}ds\right|\nonumber\\=& 
M\left|\int\limits_{0}^{t}\sum\limits_{k=1}^{\infty}\hat P_k\left(\frac{z}{||z||_{\infty}}\right)s^{k-1} ds\right|\nonumber \\=& 
M\left|\int\limits_{0}^{t}\sum\limits_{k=1}^{\infty}\hat P_k\left(\frac{z}{||z||_{\infty}}\right)\xi^{k-1}e^{ik\theta} d\xi\right|\nonumber\\ \leq&
M\int\limits_{0}^{|t|} \sum\limits_{k=1}^{\infty}\left|\hat P_k\left(\frac{z}{||z||_{\infty}}\right)\right|\xi^{k-1}d\xi\nonumber\\\leq &
2M\sum\limits_{k=1}^{\infty} \binom{k+n-1}{n-1}\frac{|t|^{k}}{k}.\nonumber
\end{align}

Thus from (\ref{Th8:1.7}) and (\ref{Th8:1.10}), we obtain
\begin{align}\label{Th8:1.11}
\left|\frac{d \tilde F_{\varepsilon}(t)}{d t}\right|=\left|\frac{d \tilde h(t)}{d t}+\varepsilon \frac{d \tilde g(t)}{d t}\right|\leq \left|\sum\limits_{k=1}^nz_k/||z||_{\infty}\right|+2M\sum\limits_{k=1}^{\infty} \binom{k+n-1}{n-1}\frac{|t|^{k}}{k}
\end{align}
and
\begin{align}\label{Th8:1.12}
\left|\frac{d \tilde F_{\varepsilon}(t)}{d t}\right|=\left|\frac{d \tilde h(t)}{d t}+\varepsilon \frac{d \tilde g(t)}{d t}\right|\geq \left|\sum\limits_{k=1}^nz_k/||z||_{\infty}\right|-2M\sum\limits_{k=1}^{\infty} \binom{k+n-1}{n-1}\frac{|t|^{k}}{k}.
\end{align}
 
Since $\varepsilon\;(|\varepsilon|=1)$ is arbitrary, it follows from (\ref{Th8:1.11}) and (\ref{Th8:1.12}) that
\begin{align}\label{Th8:1.13}
\left|\frac{d \tilde h(t)}{d t}\right|+\left|\frac{d \tilde g(t)}{d t}\right|\leq \left|\sum\limits_{k=1}^nz_k/||z||_{\infty}\right|+2M\sum\limits_{k=1}^{\infty} \binom{k+n-1}{n-1}\frac{|t|^{k}}{k}
\end{align}
and
\begin{align}\label{Th8:1.14}
\left|\frac{d \tilde h(t)}{d t}\right|-\left|\frac{d \tilde g(t)}{d t}\right|\geq \left|\sum\limits_{k=1}^nz_k/||z||_{\infty}\right|-2M\sum\limits_{k=1}^{\infty} \binom{k+n-1}{n-1}\frac{|t|^{k}}{k}.
\end{align}

Let $\Gamma$ be the radial segment from $0$ to $||z||_{\infty}$. Then from (\ref{Th8:1.15}) and (\ref{Th8:1.13}), we have
\begin{align}\label{Th8:1.16}
|\tilde f(\|z\|_{\infty})|=&\left|\int_{\Gamma} \frac{\partial \tilde f(s)}{\partial s}d s+\frac{\partial \tilde f(s)}{\partial \ol s}d \ol s\right|\\\leq&
\int\limits_{0}^{||z||_{\infty}}\left\lbrack \left|\frac{d \tilde h(s)}{d s}\right|+\left|\frac{d \tilde g(s)}{d s}\right|\right\rbrack |d s|\nonumber\\\leq&
\int\limits_{0}^{||z||_{\infty}}\left(\left|\sum\limits_{k=1}^nz_k/||z||_{\infty}\right|+2M\sum\limits_{k=1}^{\infty} \binom{k+n-1}{n-1}\frac{\xi^{k}}{k}\right)d\xi\nonumber\\=&\left|\sum\limits_{k=1}^nz_k/||z||_{\infty}\right|\times||z||_{\infty}+2M\sum\limits_{k=1}^{\infty} \binom{k+n-1}{n-1}\frac{||z||_{\infty}^{k+1}}{k(k+1)}\nonumber\\=&
\left|\sum\limits_{k=1}^nz_k\right|+2M\sum\limits_{m=2}^{\infty} \binom{m+n-2}{n-1}\frac{||z||_{\infty}^{m}}{m(m-1)}.\nonumber
\end{align}

Consequently from (\ref{Th8:1.15}) and (\ref{Th8:1.16}), we get
\begin{align*}
|f(z)|\leq\left|\sum\limits_{k=1}^nz_k\right|+2M\sum\limits_{m=2}^{\infty} \binom{m+n-2}{n-1}\frac{\|z\|_{\infty}^{m}}{m(m-1)},
\end{align*}
which also holds for $z=0$.

Again from (\ref{Th8:1.15}) and (\ref{Th8:1.14}), we have
\begin{align*}
|\tilde f(\|z\|_{\infty})|=\left|\int_{\Gamma} \frac{\partial \tilde f(s)}{\partial s}d s+\frac{\partial \tilde f(s)}{\partial \ol s}d \ol s\right|\geq&
\int\limits_{0}^{||z||_{\infty}}\left\lbrack \left|\frac{d \tilde h(s)}{d s}\right|-\left|\frac{d \tilde g(s)}{d s}\right|\right\rbrack |d s|\nonumber\\\geq&
\left|\sum\limits_{k=1}^nz_k\right|-2M\sum\limits_{m=2}^{\infty} \binom{m+n-2}{n-1}\frac{||z||_{\infty}^{m}}{m(m-1)}.\nonumber
\end{align*}
Consequently
\begin{align*}
|f(z)|\geq\left|\sum\limits_{k=1}^nz_k\right|-2M\sum\limits_{m=2}^{\infty} \binom{m+n-2}{n-1}\frac{\|z\|_{\infty}^{m}}{m(m-1)},
\end{align*}
which also holds for $z=0$.

Thus the desired inequalities are established.\vspace{2mm}

To show the inequality $(i)$ is sharp, we consider the function $f_2(z)$ defined by
\begin{align*}
	f_2(z)=\sum\limits_{j=1}^2 z_j+\sum\limits_{m=2}^{\infty}\sum\limits_{|\beta|=m}a_{\alpha}z^{\alpha},
\end{align*}
where
\begin{align*}
	a_{\alpha}=-\frac{2M}{(m-1)(m+1)}
\end{align*}
for $\alpha=(\alpha_1,\ldots,\alpha_n)$ such that $|\alpha|=m\geq 2$. It is easy to verify that $f_2\in \mathcal{P}_{\mathcal{H}_2^0}(M)$. Now for the point $z=(r,r)$, where $0<r<1$, we find that
\begin{align*}
|f_2(z)|=f_2(z)=&2r-2M\sum\limits_{m=2}^{\infty}\frac{\|z\|_{\infty}^m}{(m-1)(m+1)}\sum\limits_{|\alpha|=m} 1\\=&
2r-2M\sum\limits_{m=2}^{\infty}\frac{r^m}{m-1}\\=&
2r\left(1+M\log (1-r)\right)\\=&
\left|\sum\limits_{k=1}^nz_k\right|-2M\sum\limits_{m=2}^{\infty} \binom{m+n-2}{n-1}\frac{\|z\|_{\infty}^{m}}{m(m-1)}
\end{align*}
for $n=2$. Therefore the inequality $(i)$ is sharp.\vspace{1.2mm}

Next to show the inequality $(ii)$ is sharp, we consider the function $f_3(z)$ defined by
\begin{align*}
	f_3(z)=\sum\limits_{j=1}^n z_j+\sum\limits_{m=2}^{\infty}\sum\limits_{|\alpha|=m}b_{\alpha}z^{\alpha},
\end{align*}
where
\begin{align*}
	b_{\alpha}=\frac{2M\binom{m+n-2}{n-1}}{m(m-1)\binom{m+n-1}{n-1}}
\end{align*}
for $\alpha=(\alpha_1,\ldots,\alpha_n)$ such that $|\alpha|=m\geq 2$. It is easy to verify that $f_3\in \mathcal{P}_{\mathcal{H}_n^0}(M)$.
For the point $z=(r,r,\ldots,r)$, where $r<1$, we have
\begin{align*}
	|f_3(z)|=\left|\sum\limits_{k=1}^nz_k\right|+M\sum\limits_{m=2}^{\infty}\binom{m+n-2}{n-1} \frac{||z||_{\infty}^m}{m(m-1)},
\end{align*}
which shows that the inequality $(ii)$ is sharp. 
\end{proof}

\vspace{5mm}

\noindent\textbf{Conflict of interest:} The authors declare that there is no conflict  of interest regarding the publication of this paper.\vspace{1.2mm}

\noindent {\bf Funding:} Not Applicable.\vspace{1.2mm}

\noindent\textbf{Data availability statement:}  Data sharing not applicable to this article as no datasets were generated or analysed during the current study.\vspace{1.2mm}

\noindent {\bf Authors' contributions:} All the authors have equal contributions in preparation of the manuscript.

\end{document}